\documentclass{article}
\usepackage[utf8]{inputenc}
\usepackage{amsfonts}
\usepackage{amsmath}
\usepackage{graphicx}
\usepackage{amssymb}
\usepackage{graphicx}
\usepackage{amsthm}
\usepackage{t1enc}
\usepackage{subfig}
\usepackage{MnSymbol}
\usepackage{mathrsfs}
\usepackage{bbm}
\usepackage{mathtools}
\usepackage[shortlabels]{enumitem}
\usepackage{tikz}
\usetikzlibrary{matrix}
\usepackage{tikz-cd}
\usepackage[cal=boondoxo]{mathalfa}
\usetikzlibrary{angles,quotes}
\usepackage{url}

\usepackage[a4paper ]{geometry}

\newtheorem{theorem}{Theorem}[section]
\newtheorem{lemma}[theorem]{Lemma}
\newtheorem{proposition}[theorem]{Proposition}

\newtheorem{definition}[theorem]{Definition}

\newtheorem{conjecture}[theorem]{Conjecture}

\def\ord{\textup{\textrm{ord}}}
\def\N{\textup{\textrm{N}}}
\def\rec{\textup{\textrm{rec}}}

\def\Gal{\textup{\textrm{Gal}}}

\def\Xint#1{\mathchoice
{\XXint\displaystyle\textstyle{#1}}{\XXint\textstyle\scriptstyle{#1}}%
{\XXint\scriptstyle\scriptscriptstyle{#1}}{\XXint\scriptscriptstyle\scriptscriptstyle{#1}}\!\int}

\def\XXint#1#2#3{{\setbox0=\hbox{$#1{#2#3}{\int}$}\vcenter{\hbox{$#2#3$}}\kern-.5\wd0}}

\def\multint{\Xint\times}

\usepackage[colorlinks,
pdfpagelabels,
pdfstartview = FitH,
bookmarksopen = true,
bookmarksnumbered = true,
linkcolor = black,
plainpages = false,
hypertexnames = false,
citecolor = black, urlcolor=black]{hyperref}

\tikzset{mynode/.style={font=\footnotesize,inner sep=0pt,text=black}
}

\title{On the Root of Unity Ambiguity in a Formula for the Brumer--Stark Units}
\author{Matthew H. Honnor}
\date{\today}

\begin{document}

\maketitle

\begin{abstract}

We prove a conjectural formula for the Brumer--Stark units. Dasgupta--Kakde have shown the formula is correct up to a bounded root of unity. In this paper we resolve the ambiguity in their result. We also remove an assumption from Dasgupta--Kakde's result on the formula.

\end{abstract}

\tableofcontents

\section{Introduction}

Let $F$ be a number field with ring of integers $\mathcal{O}=\mathcal{O}_F$. Let $\mathfrak{p}$ be a prime ideal of $F$, lying above a rational prime $p$, and let $H$ be a finite abelian extension of $F$ such that $\mathfrak{p}$ splits completely in $H$. Write $G=\Gal(H/F)$ and let $\sigma \in G$. In 1981, Tate proposed the Brumer--Stark conjecture, \cite[Conjecture $5.4$]{MR656067}, stating, for each $\sigma \in G$, the existence of a $\mathfrak{p}$-unit $u_\mathfrak{p}(\sigma)$ in $H$, the Brumer--Stark unit. This unit has $\mathfrak{P}$-order equal to the value of a partial zeta function at 0, for a prime ideal $\mathfrak{P}$ of $H$, lying above $\mathfrak{p}$. This partial zeta function depends on $\sigma$. The unit $u_\mathfrak{p}(\sigma)$ is only non-trivial when $F$ is totally real and $H$ is totally complex containing a complex multiplication (CM) subfield, we assume this for the remainder of the paper. The ground-breaking work of Dasgupta--Kakde and Dasgupta--Kakde--Silliman--Wang in \cite{Brumerstark} and \cite{BrumerstarkoverZ} has shown that the Brumer--Stark conjecture holds.

In \cite[Definition 3.18]{MR2420508}, Dasgupta constructed explicitly, in terms of the values of Shintani zeta functions at $s=0$, an element $u_1(\sigma) \in F_\mathfrak{p}^\ast$. In \cite[Conjecture 3.21]{MR2420508}, Dasgupta conjectured that this unit is equal to the Brumer--Stark unit. This formula has recently been shown to be correct up to a bounded root of unity by Dasgupta--Kakde in \cite[Theorem 1.6]{intgrossstark}. The key ingredient in the proof of the above theorem is Dasgupta--Kakde's proof of the $p$-part of the integral Gross--Stark conjecture, which Dasgupta proved in \cite[Theorem 5.18]{MR2420508} implies the result. The work in \cite[Theorem 5.18]{MR2420508} requires the assumption that $S$, a finite set of places of $F$ required for the Brumer--Stark conjecture, contains a prime with associated Frobenius element equal to complex conjugation in $H$. The first result of this paper is the removal of this assumption from \cite[Theorem 1.6]{intgrossstark}. For this we need to work with a different formula for the Brumer--Stark units.

There have been two other formulas conjectured for the Brumer--Stark units. These two formulas are cohomological in nature and were conjectured by Dasgupta--Spie\ss \ in \cite{MR3861805} and \cite{MR3968788}. We denote these formulas by $u_2(\sigma)$ and $u_3(\sigma)$, respectively. In \cite[Conjecture 6.1]{MR3861805} and \cite[Remark 4.5]{MR3968788}, respectively, $u_2(\sigma)$ and $u_3(\sigma)$ are conjectured to be equal to the Brumer--Stark unit.  Recent joint work of the author with Dasgupta in \cite{equivformulas} has proved that these three conjectural formulas for the Brumer--Stark units are equal. Following from the equality of the formulas, shown in \cite{equivformulas}, Dasgupta--Kakde's work in \cite{intgrossstark} implies that $u_2(\sigma)$ and $u_3(\sigma)$ are also equal to $u_\mathfrak{p}(\sigma)$ up to bounded a root of unity. 

To remove the assumption on $S$ from \cite[Theorem 1.6]{intgrossstark} we work with $u_2(\sigma)$. The advantage of working with this formula is that it is defined for any finite abelian extension $H/F$ in which $\mathfrak{p}$ splits completely. In particular, it does not require that $H$ is totally complex. Furthermore, \cite[Proposition 6.3]{MR3861805} shows that the formula is trivial if $H$ is not totally complex. This along with the other functorial properties of the formula allows us, in \S3, to remove this assumption. 

The main work of this paper is considering the root of unity ambiguity in \cite[Theorem 1.6]{intgrossstark}. We show that the formula is in fact correct without any ambiguity. We work with $u_2(\sigma)$ and use the equality of the formulas to obtain these results for $u_1(\sigma)$ and $u_3(\sigma)$ as well.

The strategy for showing our main result takes inspiration from Dasgupta's proof of \cite[Theorem 5.18]{MR2420508}. In this theorem, Dasgupta proves that the integral Gross--Stark conjecture implies that $u_1(\sigma)$ is equal to $u_\mathfrak{p}(\sigma)$ up to a root of unity. We are able to modify the proof of this theorem to show that, for each odd rational prime $l$, the integral Gross--Stark conjecture implies that $u_2(\sigma)$ is equal to $u_\mathfrak{p}(\sigma)$ in $F_\mathfrak{p}/D_l $. Where $D_l \subset F_\mathfrak{p}$ is such that $D_l$ does not contain any roots of unity with order divisible by $l$. We then show, using the fact that $u_2(\sigma)$ equals $u_\mathfrak{p}(\sigma)$ up to a root of unity, that $u_2(\sigma)$ equals $u_\mathfrak{p}(\sigma)$ up to a $2$-power root of unity. 

Removing the $2$-power root of unity ambiguity requires further work since the strategy used for the odd primes adds in an additional $2$-power root of unity in the process of controlling the other roots of unity. Instead we show that, the $2$-part of the integral Gross--Stark conjecture implies that, for any $\sigma, \tau \in G$, we have $u_2(\sigma)u_2(\tau)^{-1}=u_\mathfrak{p}(\sigma)u_\mathfrak{p}(\tau)^{-1}$ in $F_\mathfrak{p}/D_2 $. Where $D_2 \subset F_\mathfrak{p}$ is such that $D_2$ does not contain any roots of unity with order divisible by $2$.  This, combined with the result away from $2$, shows that the ambiguity between $u_2(\sigma)$ and $u_\mathfrak{p}(\sigma)$ is independent of $\sigma \in G$. We then utilise the norm compatibility property, for the Brumer--Stark units and the formulas, to allow us to remove this ambiguity.

We remark here that the our proofs are only able to work with $u_2(\sigma)$ due to the nature of the cohomological formula and we are currently unable to obtain the result for $u_1(\sigma)$ and $u_3(\sigma)$ without first knowing $u_1(\sigma)=u_2(\sigma)=u_3(\sigma)$. The reasons for this are made more precise in \S 2. 

In \cite[Corollary 4.3]{MR2336038}, Burns shows that the integral Gross--Stark conjecture follows from the minus part of the eTNC for finite CM extensions of totally real fields. From now on, we refer to this conjecture as $\text{eTNC}^-$. 

The recent work of Bullach--Burns--Daoud--Seo in \cite[Theorem B]{etnc} proves $\text{eTNC}^-$ away from $2$ for finite CM extensions of totally real fields, a key imput for their work is the proof of the Brumer--Stark conjecture away from $2$ in \cite{Brumerstark}. Furthermore, in \cite{etnconZ}, Dasgupta--Kakde--Silliman prove $\text{eTNC}^-$ over $\mathbb{Z}$, this result uses the proof of the Brumer--Stark conjecture over $\mathbb{Z}$ in \cite{BrumerstarkoverZ}. We remark that the argument given in \cite{etnconZ} adopts the central strategy developed in \cite{etnc}. As noted above, this implies that the integral Gross--Stark conjecture holds. Thus we can show that the formulas are equal to the Brumer--Stark unit.

\subsection{Acknowledgements}

The author would like to thank Samit Dasgupta and Michael Spie\ss \ for generously giving their time and for many stimulating discussions during the course of this research. In particular, he would like to thank Michael Spie\ss \ for his help with the argument given in \S3. He would also like to thank Mahesh Kakde for suggesting this problem and, along with Dominik Bullach, for many useful conversations. The author wishes to acknowledge the financial support of the Heilbronn Institute for Mathematical Research and Imperial College London.

\section{Brumer--Stark units}

Let $R$ denote a finite set of places of $F$ such that $R$ contains the archimedean places, $\mathfrak{p} \nin R$, and $R$ contains the places that are ramified in $H$. We let $S = R \cup \{ \mathfrak{p} \}$. We also denote $G=\Gal(H/F)$. We fix this notation throughout the paper.

\begin{definition}
For $\sigma \in G$, we define the \textbf{partial zeta function}
\begin{equation}\label{pzf}
    \zeta_{R}(\sigma, s)= \sum_{\substack{(\mathfrak{a},R)=1 \\ \sigma_\mathfrak{a}=\sigma}} \N\mathfrak{a}^{-s}.
\end{equation}
Here, the sum is over all non-zero integral ideals, $\mathfrak{a}\subset \mathcal{O}_F$, that are relatively prime to the elements of $R$ and with associated Frobenius element, $\sigma_\mathfrak{a} \in G$, equal to $\sigma$.
\end{definition}

Note that the series (\ref{pzf}) converges for $\text{Re}(s)>1$ and has meromorphic continuation to $\mathbb{C}$, regular outside $s=1$. The partial zeta functions associated to the sets of primes $R$ and $S$
are related by the formula
\[ \zeta_{S}(\sigma, s)= (1-\N\mathfrak{p}^{-s}) \zeta_{R}(\sigma,s). \]
If $L$ is a finite abelian extension of $F$ and $\sigma \in \Gal(L/F)$, we use the notation $\zeta_{R}(L/F, \sigma, s)$ for the partial zeta function defined as above but with the equality $\sigma_\mathfrak{a}=\sigma$ being viewed in $\Gal(L/F)$.

\begin{definition}\label{updefn}
Define the group
\begin{equation*}
U_{\mathfrak{p}}=\{ u \in H^\ast : \ \mid u \mid_\mathfrak{P}=1 \ \text{if} \ \mathfrak{P} \ \text{does not divide} \ \mathfrak{p} \}. 
\end{equation*}
\end{definition}

Here, $\mathfrak{P}$ ranges over all finite and archimedean places of $H$; in particular, complex conjugation in $H$ acts as an inversion on $U_{\mathfrak{p}}$. We introduce an auxiliary finite set $T$ of primes of $F$, disjoint from $S$. The partial zeta function associated to the sets $R$ and $T$ is defined by the group ring equation
\begin{equation}\label{pzfT}
    \sum_{\sigma \in G} \zeta_{R,T}(\sigma, s)[\sigma]= \prod_{\eta \in T}(1-[\sigma_\eta]\N\eta^{1-s}) \sum_{\sigma \in G} \zeta_{R}(\sigma, s)[\sigma].
\end{equation}
We also assume that the set $T$ contains at least two primes of different residue characteristic or at least one prime $\eta$ with absolute ramification degree at most $l-2$ where $\eta$ lies above $l$. We have this assumption on $T$ from now on. It follows from work of Deligne--Ribet \cite{MR579702} and Cassou-Nogu\'es \cite{MR524276} that, $\zeta_{R,T}(L/F, \sigma, 0) \in \mathbb{Z}$, for any finite abelian extension $L/F$ unramified outside $R$ and any $\sigma \in \Gal(L/F)$. The following conjecture was first stated by Tate and called the Brumer--Stark conjecture, \cite[Conjecture $5.4$]{MR656067}. We present the formulation given by Gross.
 
\begin{conjecture}[Conjecture $7.4$, \cite{MR931448}]\label{Conj2.5}
Let $\mathfrak{P}$ be a prime in $H$ above $\mathfrak{p}$. There exists an element $u_T \in U_{\mathfrak{p}}$ such that $u_T \equiv 1 \pmod{T}$, and for all $\sigma \in G$, we have 
\begin{equation}\label{uteqn}
    \ord_\mathfrak{P}(\sigma( u_T))= \zeta_{R,T}(H/F, \sigma,0).
\end{equation}
\end{conjecture}

Our assumption on $T$ implies that there are no nontrivial roots of unity in $H$ that are congruent to $1$ modulo $T$. Thus, the $\mathfrak{p}$-unit, if it exists, is unique. Note also that our $u_T$ is actually the inverse of the $u$ in \cite[Conjecture $7.4$]{MR931448}.

The conjectural element $u_T \in U_{\mathfrak{p}}$ satisfying Conjecture \ref{Conj2.5} is called the Brumer--Stark unit for the data $(S,T,H,\mathfrak{P})$. This conjecture has been recently proved, away from $2$, by Dasgupta--Kakde in \cite[Theorem 1.2]{Brumerstark} and over $\mathbb{Z}$ by Dasgupta--Kakde--Silliman--Wang in \cite[Theorem 1.1]{BrumerstarkoverZ}. It is convenient for us to work with the following element of $H^\ast[G]$. We define 
\[ u_\mathfrak{p} = \sum_{\sigma \in G} u_\mathfrak{p}(\sigma) \otimes \sigma^{-1} = \sum_{\sigma \in G} \sigma( u_T) \otimes \sigma^{-1}  \in H^\ast[G] . \]

There have been three formulas conjectured for the Brumer--Stark unit $u_\mathfrak{p}$ in $F_\mathfrak{p}^\ast \otimes \mathbb{Z}[G]$. The first, by Dasgupta in \cite{MR2420508}, is a $p$-adic analytic formula which we denote by $u_1$. The other two formulas were defined by Dasgupta--Spie\ss \ in \cite{MR3861805} and \cite{MR3968788}, we denote these by $u_2$ and $u_3$ respectively. Both these formulas are cohomological in nature and are defined using an Eisenstein cocyle. Joint work of the author with Dasgupta in \cite{equivformulas} proves that these three formulas are equal. In this paper, we work with the cohomological formula defined in \cite{MR3861805} and justify this choice after Proposition \ref{prop6.3}. To stay consistent with the notation in \cite{equivformulas}, we denote this formula by 
\[ u_2 = \sum_{\sigma \in G} u_2(\sigma) \otimes \sigma^{-1} \in F_\mathfrak{p}^\ast \otimes \mathbb{Z}[G] . \]
The formulation of $u_2$ depends on the choice of an infinite place $v$ of $F$. We fix this $v$ from now on. Dasgupta--Spie\ss \ have shown the following proposition concerning properties of their cohomological construction in \cite{MR3861805}. We require this proposition to prove the main results of this paper. We note that this proposition contains all the information we require and therefore in this paper we avoid working with the explicit definition of $u_2$. For the explicit construction, we refer readers to \cite{MR3861805}. We write
\[ \overline{T} = \{ \mathfrak{t} \ \text{prime of} \ F :  \mathfrak{t} \mid q  \ \text{where, for some} \ \mathfrak{q} \in T, \  \mathfrak{q} \mid q   \} . \]

\begin{proposition}[Proposition 6.3, \cite{MR3861805}]\label{prop6.3}
We have the following properties for $u_2$.
\begin{enumerate}[a)]
    \item For $\sigma \in G$ we have $\ord_\mathfrak{p}(u_2(\sigma)) = \zeta_{R,T}(\sigma, 0)$.
    
    \item Let $L/F$ be an abelian extension with $L \supseteq H$ and put $\mathcal{G}= \Gal(L/F)$. Assume that $L/F$ is unramified outside $S$ and that $\mathfrak{p}$ splits completely in $L$. Then we have
    \[ u_2(\sigma) = \prod_{\tau \in \mathcal{G}, \tau \mid_H = \sigma} u_2(L/F , \tau). \]
    
    \item Let $\mathfrak{r}$ be a nonarchimedean place of $F$ with $\mathfrak{r} \nin S \cup \overline{T}$. Then we have 
    \[ u_2(S \cup \{ \mathfrak{r} \} , \sigma) = u_2(S ,  \sigma) u_2(S , \sigma_\mathfrak{r}^{-1} \sigma)^{-1}. \]
    
    \item Assume that $H$ has a real archimedean place $w \nmid v$. Then $u_2( \sigma) =1$ for all $\sigma \in G$.
    
    \item Let $L/F$ be a finite abelian extension of $F$ containing $H$ and unramified outside $S$. Then we have 
\[ \rec_\mathfrak{p}(u_2( \sigma))= \prod_{\substack{\tau \in \Gal(L/F) \\ \tau \mid_H = \sigma^{-1}  }} \tau^{\zeta_{S,T}(L/F, \tau^{-1},0)}. \]
\end{enumerate}
\end{proposition}

In the proposition above, we note the field extension or set of primes which $u_2$ is associated with where necessary. If nothing is noted, then $u_2$ is assumed to be associated with $H/F$ and $S$.

The key parts of the above proposition we require are $\textit{c)}$ and $\textit{e)}$. We note that $\textit{e)}$ has also been shown for $u_1$ in \cite{MR2420508}. However, $\textit{c)}$ has only been shown for $u_2$. The closest to this that one has for one of the other formulas is for $u_1$ where in \cite{MR2420508} Dasgupta proves that, under certain conditions, $\textit{c)}$ holds up to a root of unity for $u_1$. Since the aim of this paper is to reduce the root of unity ambiguity, we are therefore required to work with $u_2$. A further advantage of using $u_2$ is that it is formulated for all finite abelian extensions $H/F$ in which $\mathfrak{p}$ splits completely. In particular we don't require that $H$ is CM or even totally complex. Then \textit{d)} of Proposition \ref{prop6.3} gives that the formula is trivial if $H$ is not totally complex (away from $v$). We have the following conjecture due to Dasgupta--Spie\ss .
\begin{conjecture}[Conjecture 6.1, \cite{MR3861805}]\label{bsformconj}
We have a formula for the Brumer--Stark unit. More precisely,  
\[ u_2 = u_\mathfrak{p} . \]
\end{conjecture}

Recent work of Dasgupta--Kakde has proved this conjecture, with $u_1$ replacing $u_2$, up to a root of unity under some mild assumptions. In particular, they prove the following theorem. Here and from this point on we let $p$ denote the rational prime such that $ \mathfrak{p} \mid p$.

\begin{theorem}[Theorem 1.6, \cite{intgrossstark}]\label{thmforform}
Suppose that
\begin{gather}
    \text{there exists} \ \mathfrak{q} \in S \ \text{where} \ \mathfrak{q} \ \text{is a prime of} \ F \ \text{that is unramified in} \ H \nonumber \\ \text{and whose associated Frobenius} \ \sigma_\mathfrak{q} \ \text{is a complex conjugation in} \ H,
    \label{assumptiononcc}
\end{gather}
and
\begin{equation}\label{assumptionforDK}
    p \ \text{is odd and} \ H \cap F(\mu_{p^\infty}) \subset H^+, \ \text{the maximal totally real subfield of} \ H.
\end{equation}
Then, Conjecture \ref{bsformconj} for $u_1$ holds up to multiplication by a root of unity in $F_\mathfrak{p}^\ast$. i.e.,
\[ u_1 = u_\mathfrak{p} \ \text{in} \ (F_\mathfrak{p}^\ast / \mu(F_\mathfrak{p}^\ast)) \otimes \mathbb{Z}[G] . \]
Here, we write $\mu(F_\mathfrak{p}^\ast)$ for the group of roots of unity of $F_\mathfrak{p}^\ast$.
\end{theorem}

The key ingredient in the proof of the above theorem is Dasgupta--Kakde's proof of the $p$-part of the integral Gross--Stark conjecture. By the main result of \cite{equivformulas} (i.e., that $u_1=u_2=u_3$) we have that Theorem \ref{thmforform} holds for $u_2$ and $u_3$ as well. We also remark that by following the arguments of \cite{MR2420508} and using Proposition \ref{prop6.3} one can show that Theorem \ref{thmforform} holds for $u_2$ without using \cite{equivformulas}. However, it is not currently possible to prove Theorem \ref{thmforform} for $u_3$ without the work in \cite{equivformulas}.

In this paper we remove the assumption (\ref{assumptiononcc}) from Theorem \ref{thmforform} and reduce the root of unity ambiguity. In particular, the main result of this paper is the following theorem.

\begin{theorem}\label{mainthmforpaper}
Suppose that (\ref{assumptionforDK}) holds. Then, Conjecture \ref{bsformconj} holds. I.e.,
\[ u_2 = u_\mathfrak{p}  . \]
\end{theorem}

The proof of the above theorem is done in three steps. In \S3 we show how one can remove the assumption (\ref{assumptiononcc}) from Theorem \ref{thmforform} and provide the lemma which allows us not to require it in our work. When considering the root of unity ambiguity we begin by reducing, in \S 5, the ambiguity to a $2$-power root of unity. In \S6 we then remove this possible $2$-power root of unity and thus complete our proof of Theorem \ref{mainthmforpaper}.


\section{Action by complex conjugation}

In this section we show how one can remove the assumption (\ref{assumptiononcc}) from Theorem \ref{thmforform}. Furthermore, the work in this section allows for the proof of Theorem \ref{mainthmforpaper} without such an assumption.

The assumption (\ref{assumptiononcc}) appears only in one place in the proof of Theorem \ref{thmforform}. Namely it is used in the proof of \cite[Theorem 5.18]{MR2420508} which shows that the $p$-part of the integral Gross--Stark conjecture implies Theorem \ref{thmforform}. We briefly review the ideas of \cite[Theorem 5.18]{MR2420508}, this is made more explicit in the next sections where we adapt these ideas to remove the root of unity ambiguity. Dasgupta shows that it is possible to find a finite number of primes $\mathfrak{r}_1, \dots , \mathfrak{r}_s$, whose associated Frobenius elements are complex conjugation, such that the $p$-part of the integral Gross--Stark conjecture implies, 
\[ u_1(S\cup \{ \mathfrak{r}_1 , \dots , \mathfrak{r}_s \})= u_\mathfrak{p}(S\cup \{ \mathfrak{r}_1 , \dots , \mathfrak{r}_s \}) \quad \text{in} \quad (F_\mathfrak{p}^\ast/ \mu( F_\mathfrak{p}^\ast)) \otimes \mathbb{Z}[G]. \]
It is then required to show that this equality with additional places in $S$ implies the equality without these additional places. Let $c$ denote the complex conjugation in $H$ and let $\mathfrak{r}$ be a prime of $F$ with $\sigma_\mathfrak{r}=c$. Since complex conjugation acts as inversion on $U_\mathfrak{p}$ one can calculate
\[ u_\mathfrak{p}(S \cup \{ \mathfrak{r} \}, \sigma )= u_\mathfrak{p}(S, \sigma) u_\mathfrak{p}(S, \sigma \sigma_\mathfrak{r}^{-1})^{-1} = u_\mathfrak{p}(S, \sigma)^2 .   \]
In the proof of \cite[Theorem 5.18]{MR2420508} it is shown that the above equality, with $u_\mathfrak{p}$ replaced by $u_1$ holds if $S$ contains a prime whose associated Frobenius is complex conjugation. This is the only point where the the assumption (\ref{assumptiononcc}) is required. To remove this assumption we prove the following theorem. Note that there are no assumptions on $S$ required here. 

\begin{theorem}\label{t:cc}
    Let $\mathfrak{r}$ be a prime of $F$ with $\sigma_\mathfrak{r}=c$. Then,
    \[ u_1(S \cup \{ \mathfrak{r} \}, \sigma )=  u_1(S, \sigma)^2 . \]
\end{theorem}

We begin with the following lemma. The author would like to credit Michael Spie\ss \ with central idea for the proof of the lemma below. Recall that in the definition of $u_2$ we fix $v$, an infinite place of $F$.

\begin{lemma}\label{l:cc}
    Let $c \in G$ denote a complex conjugation in $H$. If the complex place $w$ associated to $c$ does not lie above $v$, then
    \[ u_2(\sigma c)=u_2(\sigma)^{-1} . \]
\end{lemma}

\begin{proof}
    Let $H^{\langle c \rangle}$ denote the subfield of $H$ which is fixed by $c$. Note that $F \subseteq H^{\langle c \rangle} \subset H$. By $\textit{d)}$ of Proposition \ref{prop6.3} and since the complex place associated to $c$ does not lie above $v$,
    \[ u_2(H^{\langle c \rangle}/F, \sigma)=1, \]
    for all $\sigma \in \Gal(H^{\langle c \rangle}/F)$. Applying $\textit{b)}$ of Proposition \ref{prop6.3} with the tower of fields $H/H^{\langle c \rangle}/F$ we have
    \[ 1=u_2(H^{\langle c \rangle}/F, \sigma) = u_2(H/F,\sigma )u_2(H/F,\sigma c) . \]
    This proves the theorem.
\end{proof}

We remark that the choice of $v$ in $u_1$ refers to a choice of Shintani domain used in $u_1$; this is made clear in \cite{equivformulas}. 

\begin{proof}[Proof of Theorem \ref{t:cc}]
    In the proof of \cite[Theorem 5.18]{MR2420508} the following equality is shown,
    \[ u_1(S \cup \{ \mathfrak{r} \}, \sigma ) = u_1(S, \sigma)u_1(S, \sigma \sigma_{\mathfrak{r}}^{-1})^{-1} . \]
    In \cite{equivformulas} it is proved that $u_1=u_2$. Applying Lemma \ref{l:cc} with this equality and the fact that $\sigma_\mathfrak{r}=c$ gives the result.
\end{proof}

\section{The integral Gross--Stark conjecture}

The integral Gross--Stark conjecture or, as it is also known, Gross's tower of fields conjecture, is an integral version of the Gross--Stark conjecture. See \cite[Proposition 3.8, Conjecture 3.13]{MR656068} for the statement and further details on the Gross--Stark conjecture. Gross first stated the integral Gross--Stark conjecture in \cite{MR931448}. In this conjecture, we consider a tower of fields $L/H/F$ and as before, $F$ is totally real. We take $H$ and $L$ to be finite abelian extensions of $F$ that are CM fields such that $L$ contains $H$. Write $\mathcal{G} =\Gal(L/F)$. Recall that $S = R \cup \{ \mathfrak{p} \}$ where $\mathfrak{p}$ splits completely in $H/F$. The integral Gross--Stark conjecture gives a relationship between Brumer--Stark units and the Stickelberger element for $L/F$, $\Theta_{S,T}^{L/F}$. We define
\[ \Theta_{S,T}^{L/F} = \sum_{\sigma \in G} \zeta_{S,T} (L/F,\sigma , 0) \otimes \sigma^{-1} \in \mathbb{C}[\mathcal{G}] . \]
Letting $T$ be as in \S 2, we have $\Theta_{S,T}^{L/F} \in \mathbb{Z}[\mathcal{G}]$. We fix this choice from now on. Denote by 
\[ \rec_\mathfrak{p} : F_\mathfrak{p}^\ast \rightarrow \mathbb{A}_F^\ast \rightarrow \mathcal{G} \]
the reciprocity map of local class field theory. Since $H \subset H_\mathfrak{P} \cong F_\mathfrak{p}$, we can evaluate $\rec_\mathfrak{p}$ on $H^\ast$. Note that if $x\in H^\ast $, then $\rec_\mathfrak{p}(x) \in \Gal(L/H)$. Let $I$ denote the relative augmentation ideal associated to $\mathcal{G}$ and $G$, i.e., the kernel of the canonical projection
\[ \text{Aug}^\mathcal{G}_G : \mathbb{Z}[\mathcal{G}] \twoheadrightarrow \mathbb{Z}[G]. \]
We state the integral Gross--Stark conjecture.
\begin{conjecture}[Conjecture 7.6, \cite{MR931448}]\label{intgs}
Define 
\[ \rec_G(u_\mathfrak{p}) = \sum_{\sigma \in G} (\rec_\mathfrak{p} (u_\mathfrak{p}(\sigma))-1 ) \Tilde{\sigma}^{-1} \in I/I^2 , \]
where $\Tilde{\sigma} \in \mathcal{G}$ is any lift of $\sigma \in G$ and $u_\mathfrak{p} = \sum_{\sigma \in G} u_\mathfrak{p}(\sigma) \otimes \sigma^{-1}$ is the Brumer--Stark unit. Then
\begin{equation}\label{igsce}
    \rec_G(u_\mathfrak{p}) \equiv \Theta_{S,T}^{L/F} 
\end{equation}
in $I/I^2$.
\end{conjecture}
To see how this conjecture can provide more information about the Brumer--Stark unit, we consider the $\sigma$ component of (\ref{igsce}). We then see that Conjecture \ref{intgs} implies
\begin{equation}\label{intgsandrec}
    \rec_\mathfrak{p}(u_\mathfrak{p}(\sigma)) = \prod_{\substack{ \tau \in \mathcal{G} \\ \tau \mid_H=\sigma^{-1} }} \tau^{\zeta_{S,T}(L/F, \tau^{-1},0) }.
\end{equation}

Taking the inverse of $\rec_\mathfrak{p}$ on both sides of the above equation allows us to gain more information about the unit $u_\mathfrak{p}$. In particular, it gives us the value of $u_\mathfrak{p}(\sigma) \in F_\mathfrak{p}^\ast/ \ker(\rec_\mathfrak{p})$. We follow the ideas presented in \cite[\S 3.1]{MR2420508}. Let $\mathfrak{f}$ denote the conductor of $H/F$. To gain some more precise information, one can apply (\ref{intgsandrec}) with $L=K$ for every $H \subset K \subset H_{\mathfrak{fp}^\infty}$. Here, we define $H_{\mathfrak{fp}^\infty}$ to be the union of the narrow ray class fields, $H_{\mathfrak{fp}^m}$, of $F$ of conductor $\mathfrak{fp}^m$, for each $m \in \mathbb{Z}_{\geq 1}$. The local reciprocity map at $\mathfrak{p}$ induces an isomorphism
\[ \rec_\mathfrak{p} : F_\mathfrak{p}^\ast / \widehat{E_+(\mathfrak{f})_\mathfrak{p}} \cong \Gal(H_{\mathfrak{fp}^\infty} / H) , \]
where we write $E_+(\mathfrak{f})_\mathfrak{p}$ for the totally positive $\mathfrak{p}$-units of $F$ which are congruent to $1 \pmod{\mathfrak{f}}$. Then, $\widehat{E_+(\mathfrak{f})_\mathfrak{p}}$ denotes the closure of $E_+(\mathfrak{f})_\mathfrak{p}$ in $F_\mathfrak{p}^\ast$. Thus, we can use (\ref{intgsandrec}) to give the value of $u_\mathfrak{p}(\sigma)$ in $ F_\mathfrak{p}^\ast / \widehat{E_+(\mathfrak{f})_\mathfrak{p}}$. In \cite{MR2420508}, Dasgupta develops the methods of horizontal Iwasawa theory to further refine this kernel and shows that the $p$-part of Conjecture \ref{intgs} implies Theorem \ref{thmforform}. It is these horizontal methods that we use in \S5 to show that Theorem \ref{mainthmforpaper} follows from the combination of the $l$-part of Conjecture \ref{intgs} for every rational prime $l$.

The $p$-part of Conjecture \ref{intgs}, when $p$ is odd, has recently been proved by Dasgupta--Kakde in \cite{intgrossstark}. We give the statement of their theorem below.

\begin{theorem}[Theorem 1.4, \cite{intgrossstark}]\label{dkintgross}
Let $p$ be an odd prime and suppose that $\mathfrak{p}$ lies above $p$. The integral Gross--Stark conjecture (Conjecture \ref{intgs}) holds in $(I/I^2)\otimes \mathbb{Z}_p$.
\end{theorem}

In this paper we require the following theorem.

\begin{theorem}\label{intgsoverZ}
The integral Gross--Stark conjecture holds over $\mathbb{Z}$.
\end{theorem}

\begin{proof}
Burns has proved in \cite[Corollary 4.3]{MR2336038} that Conjecture \ref{intgs} is implied by $\text{eTNC}^-$, \cite[Theorem 1]{etnconZ} therefore completes the result.
\end{proof}

\section{Equality of the formula up to a $2$-power root of unity}\label{s:mainproof}

In this section we reduce the root of unity ambiguity in Theorem \ref{thmforform} to a $2$-power root of unity. We prove the following theorem.

\begin{theorem}\label{t:reduce to 2-power root of unity}
    Suppose that (\ref{assumptionforDK}) holds. Then, Conjecture \ref{bsformconj} holds. I.e.,
    \[ u_2 = u_\mathfrak{p} \quad \text{in} \quad (F_\mathfrak{p}^\ast/ \mu^2(F_\mathfrak{p}^\ast)) \otimes \mathbb{Z}[G] . \]
    Here, we write $\mu^2(F_\mathfrak{p}^\ast)$ for the group of 2-power roots of unity of $F_\mathfrak{p}^\ast$.
\end{theorem}
 
We show that the above theorem is implied by Theorem \ref{intgsoverZ}. To further gain more information from (\ref{intgsandrec}), we work with all possible extensions $L/H/F$. For more details on this approach, we refer readers to \cite{MR2420508}. We present here the statements which are required for our work.

Again, let $\mathfrak{f}$ be the conductor of the extension $H/F$ and write $E_+(\mathfrak{f})$ for the totally positive units of $F$ which are congruent to 1 modulo $\mathfrak{f}$. Let $\mathfrak{g}$ denote the product of the finite primes in $S$ that do not divide $\mathfrak{f} \mathfrak{p}$. Then, we define $H_S \coloneqq H_{(\mathfrak{f} \mathfrak{p} \mathfrak{g})^\infty}$. Here, $H_{(\mathfrak{f} \mathfrak{p} \mathfrak{g})^\infty}$ is the union of the narrow ray class fields $H_{\mathfrak{f}^a\mathfrak{p}^b\mathfrak{g}^c}$ for all positive integers $a,b,c$. For $v \mid \mathfrak{f} \mathfrak{g}$, let $U_{v, \mathfrak{f}}$ denote the group of elements of $\mathcal{O}_v ^\ast$ which are congruent to $1$ modulo $\mathfrak{f}\mathcal{O}_v ^\ast$; in particular, $U_{v,\mathfrak{f}}=\mathcal{O}_v^\ast$ for $v \mid \mathfrak{g}$. Let $\mathcal{U}_{\mathfrak{fg}}= \prod_{v \mid \mathfrak{f} \mathfrak{g}} U_{v, \mathfrak{f}}$.

\begin{proposition}[Proposition 3.4, \cite{MR2420508}]\label{prop2}
Theorem \ref{intgsoverZ} is equivalent to the existence of an element $u_T \in U_\mathfrak{p}$ with $u_T \equiv 1 \pmod{T}$ and 
\begin{equation*}
    ( \sigma_\mathfrak{b}( u_T),1)= \pi^{\zeta_{R,T}(H_\mathfrak{f}/F, \mathfrak{b},0)} \multint_{\mathbb{O}\times \mathcal{U}_{\mathfrak{fg}}/ \overline{E_+(\mathfrak{f})}} x \ d \mu (\mathfrak{b}, x)
\end{equation*}
in $(F_\mathfrak{p}^\ast \times \mathcal{U}_{\mathfrak{fg}}) / \overline{E_+(\mathfrak{f})}$ for all fractional ideals $\mathfrak{b}$ relatively prime to $S$.
\end{proposition}

Here, we have denoted by $\overline{E_+(\mathfrak{f})}$ the closure of $E_+(\mathfrak{f})$ diagonally embedded in $F_\mathfrak{p}^\ast \times \mathcal{U}_{\mathfrak{fg}}$. Define 
\begin{equation*}
    D(\mathfrak{f}, \mathfrak{g})= \{ x \in F_\mathfrak{p}^\ast : (x,1) \in \overline{E_+(\mathfrak{f})} \subset F_\mathfrak{p}^\ast \times \mathcal{U}_{\mathfrak{fg}} \}.
\end{equation*}
Dasgupta notes in \cite{MR2420508} that Proposition \ref{prop2} may be interpreted as stating that Conjecture \ref{intgs} is equivalent to a formula for the image of $u_T$ in $F_\mathfrak{p}^\ast / D(\mathfrak{f}, \mathfrak{g})$. The reciprocity map of class field theory induces an isomorphism
\[ \text{rec}_S : (F_\mathfrak{p}^\ast \times \mathcal{U}_{\mathfrak{fg}}) / \overline{E_+(\mathfrak{f})_\mathfrak{p}} \cong \text{Gal}(H_S/H) . \]
Here, we have defined $E_+(\mathfrak{f})_\mathfrak{p}$ as the group of totally positive $\mathfrak{p}$-units congruent to 1 modulo $\mathfrak{f}$.

\begin{proposition}\label{propGSequality}
Let $\sigma \in G$. The construction, $u_2(\sigma)$, is equal to the Brumer--Stark unit in $F_\mathfrak{p}^\ast / D(\mathfrak{f}, \mathfrak{g})$. i.e.,
\[  u_2( \sigma) \equiv u_\mathfrak{p}(\sigma) \pmod{D(\mathfrak{f}, \mathfrak{g})}. \]
\end{proposition}

\begin{proof}
We consider the unit $u_2( \sigma)$ and apply $\text{rec}_S$ to $(u_2( \sigma), 1)$. Then by $\textit{e)}$ of Proposition \ref{prop6.3}, we have 
\[ \text{rec}_S((u_2( \sigma), 1)) = \prod_{\substack{\tau \in \text{Gal}(H_S/F), \\ \tau \mid_H = \sigma^{-1}  }} \tau^{\zeta_{S,T}(H_S/F, \tau^{-1},0)} = \text{rec}_S((u_\mathfrak{p}(\sigma) ,1)), \]
where the second equality follows from (\ref{intgsandrec}) which, as we noted in \S4, follows from Conjecture \ref{intgs}. Thus, we have the result.
\end{proof}

We follow the ideas presented by Dasgupta in \cite[Lemma 5.17]{MR2420508} to control the group $D(\mathfrak{f}, \mathfrak{g})$. The lemma below is simpler than \cite[Lemma 5.17]{MR2420508} since we are only required to find conditions such that there are no roots of unity of given orders in $D(\mathfrak{f}, \mathfrak{g})$. However, in \cite[Lemma 5.17]{MR2420508} Dasgupta finds conditions, for each $m \in \mathbb{Z}_{\geq 0}$, such that $D(\mathfrak{f}, \mathfrak{g}) \subseteq \mu(F_\mathfrak{p}^\ast) \pmod{\mathfrak{p}^m}$.

Let $\mathfrak{q}$ be a prime of $F$ that is unramified in $H$ and whose associated Frobenius, $\sigma_\mathfrak{q}$, is a complex conjugation in $H$. For any rational prime $l$ we write $\mu_l(F_\mathfrak{p}^\ast) \subseteq \mu(F_\mathfrak{p}^\ast)$ for the set of roots of unity in $F_\mathfrak{p}^\ast$ with order divisible by $l$. 

\begin{lemma}\label{lemmanew5.17}
Let $l$ be an odd rational prime. There exists a finite set of prime ideals, $\{ \mathfrak{r}_1, \dots , \mathfrak{r}_s \}$, in the narrow ray class of $\mathfrak{q}$ modulo $\mathfrak{f}$ such that the group, $D(\mathfrak{f}, \mathfrak{r}_1 \dots \mathfrak{r}_s )$, does not contain any element of $\mu_{l}(F_\mathfrak{p}^\ast)$.
\end{lemma}

\begin{proof}
Let $\varepsilon \in \mu_l(F_\mathfrak{p}^\ast)$. Then there exists a prime, $\mathfrak{r}$, of $F$ such that $\mathfrak{r}$ is in the narrow ray class of $\mathfrak{q}$ modulo $\mathfrak{f}$ and such that $\varepsilon$ is not congruent to 1 modulo $\mathfrak{r}$. We note that all but finitely many of the infinite number of primes in the narrow ray class of $\mathfrak{q}$ modulo $\mathfrak{f}$ satisfy this property.

Suppose now that $\varepsilon \in D(\mathfrak{f}, \mathfrak{r})$. Then, by the definition of $D(\mathfrak{f},\mathfrak{r})$ and, in particular, the definition of $\mathcal{U}_\mathfrak{fr}$ we see that $\varepsilon \equiv 1 \pmod{\mathfrak{r}}$. This contradicts our choice of $\mathfrak{r}$. Letting the $\mathfrak{r}_i$ consist of such an ideal prime $\mathfrak{r}$ for each element $\varepsilon \in \mu_l(F_\mathfrak{p}^\ast)$ completes the proof.
\end{proof}




We are now able to prove Theorem \ref{t:reduce to 2-power root of unity}. We follow the ideas of the proof of \cite[Theorem 5.18]{MR2420508}.

\begin{proof}[Proof of Theorem \ref{t:reduce to 2-power root of unity}]
Let $\sigma \in G$. We begin by noting that the roots of unity in $F_\mathfrak{p}^\ast$ have orders that divide $p^a(\N\mathfrak{p} -1)$ for some $a \in \mathbb{Z}_{\geq 0}$. Let $l \mid p^a(\N\mathfrak{p} -1)$ be an odd prime. Let $m \in \mathbb{Z}_{\geq 1}$ be such that $l^m$ exactly divides $p^a(\N\mathfrak{p} -1)$. Let $\varepsilon_l$ be a primitive root of unity of order $l^m$ and write 
\[ \mu^l(F_\mathfrak{p}^\ast) \coloneqq \langle  \varepsilon_l \rangle \subseteq \mu(F_\mathfrak{p}^\ast), \]
for the subgroup of $\mu(F_\mathfrak{p}^\ast)$ generated by $\varepsilon_l$. This definition is independent of the choice of the primitive root of unity $\varepsilon_l$. We show that, for each $\sigma \in G$,
\begin{equation}\label{maineqntoshow}
    u_2(\sigma) = u_\mathfrak{p}(\sigma) \ \text{in} \ F_\mathfrak{p}^\ast / (\mu(F_\mathfrak{p}^\ast) /  \mu^l(F_\mathfrak{p}^\ast) ) .
\end{equation}
Repeating this for each such odd prime gives the result. Fix such a prime $l$. Let $\{ \mathfrak{r}_1, \dots , \mathfrak{r}_s \}$ be a finite set of prime ideals as in Lemma \ref{lemmanew5.17}, and let $\mathfrak{r}$ be one of the $\mathfrak{r}_i$. It follows from 
\begin{equation}\label{e:zetareln}
    \zeta_{R\cup\{\mathfrak{r}\}}(H/F, \sigma, s)= \zeta_R(H/F, \sigma, s)- \text{N}\mathfrak{r}^{-s} \zeta_R(H/F, \sigma \sigma_{\mathfrak{r}}^{-1}, s) ,
\end{equation}
that the Brumer--Stark units attached to $S$ and $S \cup \{\mathfrak{r} \}$ are related by
\[ u_\mathfrak{p}(S \cup \{ \mathfrak{r}\} ,\sigma)  = \frac{u_\mathfrak{p}(S, \sigma)}{u_\mathfrak{p}(S, \sigma \sigma_\mathfrak{r}^{-1})}=u_\mathfrak{p}(S,\sigma)^2, \]
where this last equation follows from the fact that complex conjugation acts as inversion on $U_\mathfrak{p}$. Thus, if we let $S^\prime \coloneqq S \cup \{ \mathfrak{r}_1, \dots , \mathfrak{r}_s \}$, then we inductively obtain
\begin{equation}\label{eqn67new2}
    u_\mathfrak{p}(S^\prime, \sigma) = u_\mathfrak{p}(S, \sigma)^{2^s}. 
\end{equation}
Applying Lemma \ref{l:cc} inductively, we also have 
\begin{equation}\label{eqn68new2}
    u_2(S^\prime ,  \sigma) = u_2(S ,  \sigma)^{2^s}.
\end{equation}  

It follows from Proposition \ref{propGSequality} that in $F_\mathfrak{p}^\ast / D(\mathfrak{f}, \mathfrak{r}_1 \dots  \mathfrak{r}_s)  $ we have $u_2(S^\prime , \sigma ) = u_\mathfrak{p}(S^\prime, \sigma)$. By choice of the $\mathfrak{r}_i$, we have that $\mu_{l}(F_\mathfrak{p}^\ast)$ is not contained in $D(\mathfrak{f}, \mathfrak{r}_1 \dots \mathfrak{r}_s)$. From Theorem \ref{thmforform} we know that $u_2(S^\prime , \sigma ) = u_\mathfrak{p}(S^\prime, \sigma)$ in $F_\mathfrak{p}^\ast / \mu(F_\mathfrak{p}^\ast)$. It therefore follows that
\[ u_2(S^\prime , \sigma ) = u_\mathfrak{p}(S^\prime, \sigma) \quad \text{in} \quad F_\mathfrak{p}^\ast / (\mu(F_\mathfrak{p}^\ast) / \mu^l(F_\mathfrak{p}^\ast) ) . \]
Recall that $l$ is an odd prime, it then follows from (\ref{eqn67new2}) and (\ref{eqn68new2}) that 
\[ u_2(S , \sigma ) = u_\mathfrak{p}(S, \sigma) \quad \text{in} \quad F_\mathfrak{p}^\ast / (\mu(F_\mathfrak{p}^\ast) / \mu^l(F_\mathfrak{p}^\ast) ) . \]
Thus (\ref{maineqntoshow}) holds. As we noted above, repeating this for each odd prime which divides $p^a(\N\mathfrak{p} -1)$ gives us the result.
\end{proof}

\section{Equality of the formula at $2$}

In this section we remove the remaining $2$-power root of unity ambiguity from Theorem \ref{t:reduce to 2-power root of unity}.

As in the previous section we follow an idea of Dasgupta in \cite{MR2420508} to add aditional primes into the set $S$. We require a modification of Lemma \ref{lemmanew5.17} however as we need to work with primes in any given narrow ray class.

\begin{lemma}\label{l:findprime}
    Let $\tau \in G$. There exists a finite prime $\mathfrak{r} \nin S \cup \overline{T}$ such that $\sigma_\mathfrak{r}=\tau^{-1}$ and the group $D(\mathfrak{f}, \mathfrak{r})$ does not contain any element of $\mu_2(F_\mathfrak{p}^\ast)$. Here, as before, $\mu_2(F_\mathfrak{p}^\ast)$ is the subset of $\mu(F_\mathfrak{p}^\ast)$ containing the roots of unity with order divisible by $2$.
\end{lemma}

\begin{proof}
    Let $\varepsilon \in \mu_2(F_\mathfrak{p}^\ast)$. As noted in the proof of Lemma \ref{lemmanew5.17}, for any prime idea $\mathfrak{q}$, all but finitely many of the infinite number of primes, $\mathfrak{r}$, in the narrow ray class of $\mathfrak{q}$ modulo $\mathfrak{f}$ are such that $\varepsilon$ is not congruent to $1$ modulo $\mathfrak{r}$. Since there are only a finite number of elements in $\mu_2(F_\mathfrak{p}^\ast)$, it is possible for us to choose one prime $\mathfrak{r}$ such that for any $\varepsilon \in \mu_2(F_\mathfrak{p}^\ast)$, we have that $\varepsilon$ is not congruent to $1$ modulo $\mathfrak{r}$. The result then follows in the same way as for Lemma \ref{lemmanew5.17}.
\end{proof}




We prove that the root of unity ambiguity has some independence from the Galois group $G$.

\begin{lemma}\label{l:moveterm}
    Let $\sigma , \tau \in G$. Then,
    \[ u_2(\sigma)u_2(\sigma \tau)^{-1} =  u_\mathfrak{p}(\sigma)u_\mathfrak{p}(\sigma \tau)^{-1} . \]
\end{lemma}

\begin{proof}
    It follows from Theorem \ref{t:reduce to 2-power root of unity}, that the result holds up to a $2$-power root of unity. Let $\mathfrak{r}$ be a prime ideal, chosen via Lemma \ref{l:findprime}, with respect to $\tau$. Let $S^\prime = S \cup \{ \mathfrak{r} \}$. We consider 
    \[ u_\mathfrak{p}(S^\prime, \sigma) . \]
    It follows from (\ref{e:zetareln}) and the definition of $\mathfrak{r}$, that 
    \begin{equation}\label{e:addprime1}
        u_\mathfrak{p}(S^\prime, \sigma) = u_\mathfrak{p}(S, \sigma) u_\mathfrak{p}(S, \sigma \tau)^{-1} .
    \end{equation}
    By $\textit{c)}$ of Proposition \ref{prop6.3} we have
    \begin{equation}\label{e:addprime2}
        u_2(S^\prime, \sigma) = u_2(S, \sigma) u_2(S, \sigma \tau)^{-1} . 
    \end{equation}
    Proposition \ref{propGSequality} implies that
    \[ u_\mathfrak{p}(S^\prime, \sigma) = u_2(S^\prime, \sigma) \quad \text{in} \quad F_\mathfrak{p}^\ast / D(\mathfrak{f}, \mathfrak{r})   . \]
    Recall, that we already know from Theorem \ref{t:reduce to 2-power root of unity} that $u_\mathfrak{p}(S^\prime, \sigma) = u_2(S^\prime, \sigma)$ in $F_\mathfrak{p}^\ast / \mu^2(F_\mathfrak{p}^\ast)$. Lemma \ref{l:findprime} implies that for all $\varepsilon \in \mu_2(F_\mathfrak{p}^\ast)$ we have $\varepsilon \nin D(\mathfrak{f}, \mathfrak{r})$. It therefore follows that
    \[u_\mathfrak{p}(S^\prime, \sigma) = u_2(S^\prime, \sigma).\]
    The result then follows from combining this equality with (\ref{e:addprime1}) and (\ref{e:addprime2}).
\end{proof}


We are ready to prove the main result of this paper.

\begin{proof}[Proof of Theorem \ref{mainthmforpaper}]
    From Theorem \ref{t:reduce to 2-power root of unity} we have that for each $\sigma \in G$, there exists $\gamma_\sigma \in \mu^2(F_\mathfrak{p}^\ast)$ such that
    \[ \gamma_\sigma u_2(\sigma)= u_\mathfrak{p}(\sigma) . \]
    Lemma \ref{l:moveterm} then implies that, for any $\tau \in G$,
    \[ u_2(\sigma)u_2(\sigma \tau)^{-1} =  u_\mathfrak{p}(\sigma)u_\mathfrak{p}(\sigma \tau)^{-1} = \gamma_\sigma   \gamma_{\sigma\tau}^{-1} u_2(\sigma)u_2(\sigma \tau)^{-1} . \]
    Therefore, $\gamma_\sigma =  \gamma_{\sigma\tau}$. Write $\gamma = \gamma_\sigma$, this is independent of the choice of $\sigma \in G$. Thus, by Theorem \ref{t:cc}, for any $\sigma \in G$ we have
    \[ 1= u_2(\sigma)u_2(\sigma c)= \gamma^2  u_\mathfrak{p}(\sigma) \gamma_{\sigma c}  =  \gamma_{\sigma}^2 . \]
    Here, $c \in G$ is complex conjugation. It follows that $\gamma \in \{ 1 , -1 \}$. Since 
    $D(\mathfrak{f}, \mathfrak{g}) \cap F \subseteq E_+(\mathfrak{f})$,
    for any integral ideal $\mathfrak{g}$, it is clear that $-1 \nin D(\mathfrak{f}, \mathfrak{g})$. Thus, considering Proposition \ref{propGSequality}, we have the result.


\end{proof}

\addcontentsline{toc}{section}{References}

\bibliography{bib}
\bibliographystyle{plain}

\end{document}